\documentclass[reqno]{amsart}
\usepackage{amsmath,amsthm,amssymb,amscd}

\newtheorem{theorem}{Theorem}[section]

\newtheorem{corollary}[theorem]{Corollary}
\theoremstyle{definition}

\theoremstyle{remark}
\newtheorem{remark}[theorem]{Remark}

\numberwithin{equation}{section}
\allowdisplaybreaks
\begin{document}

%-------------------------------------------------------------------------
% editorial commands: to be inserted by the editorial office
%
%\firstpage{1} \volume{228} \Copyrightyear{2004} \DOI{003-0001}
%\seriesextraline{This is the Concrete Title of this Book\br H.E. R and S.T.C. W, Eds.}
%\firstpage{1}

\title[Mex-related partition functions of Andrews and Newman]
 {Mex-related partitions and relations to ordinary partition and singular overpartitions}

\author{Rupam Barman}
\address{Department of Mathematics, Indian Institute of Technology Guwahati, Assam, India, PIN- 781039}
\email{rupam@iitg.ac.in}

\author{Ajit Singh}
\address{Department of Mathematics, Indian Institute of Technology Guwahati, Assam, India, PIN- 781039}
\email{ajit18@iitg.ac.in}

\date{September 11, 2020}

%\thanks{We thank Professor Ken Ono for many valuable comments on the article.}



\dedicatory{}
%%% ----------------------------------------------------------------------

\begin{abstract}
	In a recent paper, Andrews and Newman introduced certain families of partition functions using the minimal excludant or ``mex'' function. In this article, we study two of the families of functions Andrews and Newman introduced, namely $p_{t,t}(n)$ and $p_{2t,t}(n)$. We establish identities connecting the ordinary partition function $p(n)$ to $p_{t,t}(n)$ and $p_{2t,t}(n)$ for all $t\geq 1$. Using these identities, we prove that the Ramanujan's famous congruences for $p(n)$ are also satisfied by $p_{t,t}(n)$ and $p_{2t,t}(n)$ for infinitely many values of $t$. Very recently, da Silva and Sellers provide complete parity characterizations of $p_{1,1}(n)$ and $p_{3,3}(n)$. We prove that $p_{t,t}(n)\equiv \overline{C}_{4t,t}(n) \pmod{2}$ for all $n\geq 0$ and $t\geq 1$, where $\overline{C}_{4t,t}(n)$ is the Andrews' singular overpartition function. Using this congruence, 
	the parity characterization of $p_{1,1}(n)$ given by da Silva and Sellers follows from that of $\overline{C}_{4,1}(n)$.  We also give elementary proofs of certain congruences already proved by da Silva and Sellers.
\end{abstract}

%%% ----------------------------------------------------------------------
\maketitle
%%% ---------------------------------------------------------------------
\section{Introduction} 
For each set $S$ of positive integers the minimal excludant function (mex-function) is defined as follows: $$\text{mex}(S)=\text{min}(\mathbb{Z}_{>0}\setminus S).$$
Andrews and Newman \cite{Andrews-Newman} recently generalized this function to integer partitions. Given a partition $\lambda$ of $n$, they defined the mex-function $\text{mex}_{A, a}(\lambda)$ to be the smallest positive integer congruent to $a$ modulo $A$ that is not part of $\lambda$. They then defined $p_{A, a}(n)$ to be the number of partitions $\lambda$ of $n$ satisfying $$\text{mex}_{A, a}(\lambda)\equiv a \pmod{2A}.$$
For example, consider $n = 5$, $A = 2$, and $a = 2$. In the table below, we list the
seven partitions $\lambda$ of $5$ and the corresponding values of $\text{mex}_{2, 2}(\lambda)$ for each $\lambda$:\\
\begin{center}
	\begin{tabular}{|c|c|}
		\hline
		Partition $\lambda$&$\text{mex}_{2, 2}(\lambda)$\\
		\hline 
		5&2\\
		$4+1$&2\\
		$3+2$&4\\
		$3+1+1$&2\\
		$2+2+1$&4\\
		$2+1+1+1$&4\\
		$1+1+1+1+1$&2\\
		\hline 
	\end{tabular}
\end{center}
We see that four of the partitions of $5$ satisfy $\text{mex}_{2, 2}(\lambda)\equiv 2\pmod{4}$. Therefore, $p_{2, 2}(5)=4$. 
In \cite[Lemma 9]{Andrews-Newman}, Andrews and Newman proved that the generating function for $p_{t,t}(n)$ is given by
\begin{align}\label{gen-fun}
\sum_{n=0}^{\infty}p_{t,t}(n)q^n=\frac{1}{(q; q)_{\infty}}\sum_{n=0}^{\infty}(-1)^n q^{tn(n+1)/2}
\end{align}
and the generating function for $p_{2t,t}(n)$ is given by
\begin{align}\label{gen-fun1}
\sum_{n=0}^{\infty}p_{2t,t}(n)q^n=\frac{1}{(q; q)_{\infty}}\sum_{n=0}^{\infty}(-1)^n q^{tn^2},
\end{align}
where $\displaystyle (a; q)_{\infty}:= \prod_{j=0}^{\infty}(1-aq^j)$. 
\par The rank of a partition is the largest part minus the number of parts. The crank of a partition is the largest part of the partition if there are no ones as parts, and otherwise is the number of parts larger than the number of ones minus the numbers of ones. In \cite{Andrews-Newman}, Andrews and Newman proved that $p_{1, 1}(n)$ equals the number of partitions of $n$ with non-negative crank and $p_{3,3}(n)$ equals the number of partitions of $n$ with rank $\geq -1$. They also proved that $p_{2,1}(n)$ is equal to the number of partitions of $n$ into even parts. They further proved that $p_{4,2}(n) - p_o(n)$ equals the number of partitions of $n$ into parts congruent to
$\pm 4, \pm 6, \pm 8, \pm 10 \pmod{32}$ and $p_{6,3}(n) - p_o(n)$ equals the number of partitions of $n$ into parts congruent to $\pm 2, \pm 4, \pm 5, \pm 6, \pm 7, \pm 8 \pmod{24}$. Here $p_o(n)$ denotes the number of partitions of $n$ into odd parts. In a very recent paper \cite{BS}, we have proved that $p_{2^{\alpha},2^{\alpha}}(n)$ and $p_{3\cdot2^{\alpha}, 3\cdot2^{\alpha}}(n)$ are almost always even for all $\alpha\geq 1$. Using a result of Ono and Taguchi on nilpotency of Hecke operators, we have also found infinite families of congruences modulo $2$ satisfied by $p_{2^{\alpha},2^{\alpha}}(n)$ and $p_{3\cdot2^{\alpha}, 3\cdot2^{\alpha}}(n)$ for all $\alpha\geq 1$. 
In this article, we express $p_{t,t}(n)$ and $p_{2t,t}(n)$ in terms of the ordinary partition function $p(n)$ for all $t\geq 1$. Using our identities, we find that the partition functions $p_{t,t}(n)$ and $p_{2t,t}(n)$ satisfy the Ramanujan's famous congruences for $p(n)$ for infinitely many values of $t$. 
\par Beginning with the paper \cite{corteel2004}, Corteel and Lovejoy introduced and developed the theory of overpartitions. An overpartition of $n$ is a non-increasing sequence of natural numbers whose sum is $n$ in which the first occurrence of a number may be overlined.
In order to give overpartition analogues of Rogers-Ramanujan type theorems for the ordinary partition function with restricted successive ranks, Andrews \cite{andrews2015} defined the so-called singular overpartitions. Andrews' singular overpartition function $\overline{C}_{k, i}(n)$ counts the number of overpartitions of $n$ in which no part is divisible by $k$ and only parts $\equiv \pm i\pmod{k}$ may be overlined. For example, $\overline{C}_{3, 1}(4)=10$ with the relevant partitions being $4, \overline{4}, 2+2, \overline{2}+2, 2+1+1, \overline{2}+1+1, 2+\overline{1}+1, \overline{2}+\overline{1}+1, 1+1+1+1, \overline{1}+1+1+1$. 
For $k\geq 3$ and $1\leq i \leq \left\lfloor\frac{k}{2}\right\rfloor$, the generating function for $\overline{C}_{k, i}(n)$ is given by
\begin{align}\label{gen-fun-sop}
\sum_{n=0}^{\infty}\overline{C}_{k, i}(n)q^n=\frac{(q^k; q^k)_{\infty}(-q^i; q^k)_{\infty}(-q^{k-i}; q^k)_{\infty}}{(q; q)_{\infty}}.
\end{align}
In this article, we prove that $p_{t,t}(n)$ and $\overline{C}_{4t,t}(n)$ have the same parity for all $t\geq 1$. Using this, we find new congruences satisfied by $p_{t,t}(n)$. We also give elementary proofs of certain congruences proved by da Silva and Sellers in \cite{sellers}.
 %%%%%%%%%%%%%%%%%%%%%%%%%%%%
 \section{Mex-related partitions and relations to ordinary partition} Let $p(n)$ denote the ordinary partition function. We adopt the convention that $p(n)=0$ when $n$ is a negative integer. In the following theorem, we express $p_{t,t}(n)$ and $p_{2t,t}(n)$ in terms of $p(n)$.
 \begin{theorem}\label{thm1}
 	Let $t$ be a positive integer. Then, for all non-negative integer $n$, we have
 	\begin{align}\label{thm1.1}
 	p_{t,t}(n)= p(n)+\sum_{r=1}^{\infty}p(n-tr(2r+1))-\sum_{s=1}^{\infty}p(n-ts(2s-1))
 	\end{align}
 	and 
 	\begin{align}\label{thm1.2}
 	p_{2t,t}(n)= p(n)+\sum_{r=1}^{\infty}p(n-4tr^2)-\sum_{s=1}^{\infty}p(n-t(2s-1)^2).
 	\end{align}
 \end{theorem}
 \begin{proof}
 	We know that the generating function for the partition function $p(n)$ is given by
 	\begin{align*}
 	\sum_{n=0}^{\infty} p(n) q^{n}=\frac{1}{(q ; q)_{\infty}}.
 	\end{align*}
 	From \eqref{gen-fun}, we find that
 	\begin{align*}
 	\sum_{n=0}^{\infty}p_{t,t}(n)q^n&=\frac{1}{(q; q)_{\infty}}\sum_{n=0}^{\infty}(-1)^n q^{tn(n+1)/2}\\
 	&=\left(\sum_{n=0}^{\infty} p(n) q^{n}\right)\left(1+\sum_{r=1}^{\infty} q^{tr(2r+1)}-\sum_{s=1}^{\infty} q^{ts(2s-1)}\right)\\
 	&=\sum_{n=0}^{\infty}\left(p(n)+\sum_{r=1}^{\infty}p(n-tr(2r+1))-\sum_{s=1}^{\infty}p(n-ts(2s-1))\right)q^n.
 	\end{align*}
 	Thus, for all non-negative integer $n$, we have 
 	\begin{align}\label{eqn-NEW-1}
 	p_{t,t}(n)= p(n)+\sum_{r=1}^{\infty}p(n-tr(2r+1))-\sum_{s=1}^{\infty}p(n-ts(2s-1)).
 	\end{align}  
 	Again, from \eqref{gen-fun1}, we find that
 	\begin{align*}
 	\sum_{n=0}^{\infty}p_{2t,t}(n)q^n&=\frac{1}{(q; q)_{\infty}}\sum_{n=0}^{\infty}(-1)^n q^{tn^2}\\
 	&=\left(\sum_{n=0}^{\infty} p(n) q^{n}\right)\left(1+\sum_{r=1}^{\infty} q^{4tr^2}-\sum_{s=1}^{\infty} q^{t(2s-1)^2}\right)\\
 	&=\sum_{n=0}^{\infty}\left(p(n)+\sum_{r=1}^{\infty}p(n-4tr^2)-\sum_{s=1}^{\infty}p(n-t(2s-1)^2)\right)q^n.
 	\end{align*}
 	Thus, for all non-negative integer $n$, we have  
 	\begin{align*}
 	p_{2t,t}(n)= p(n)+\sum_{r=1}^{\infty}p(n-4tr^2)-\sum_{s=1}^{\infty}p(n-t(2s-1)^2).
 	\end{align*}
 	This completes the proof of the theorem.	
 \end{proof}
In the following theorem, we prove that $p_{t,t}(n)$ and $p_{2t,t}(n)$ satisfy Ramanujan-type congruences, and these congruences follow from those satisfied by the ordinary partition function $p(n)$.  
 \begin{theorem}\label{thm2}
 	Let $m, a\geq 1$ and $b$ be integers. Suppose  that $p(an+b)\equiv 0 \pmod m$ for all non-negative integer $n$. Then, for all $t\geq 1$, we have 
 	\begin{align*}
 	p_{at,at}(an+b)\equiv 0 \pmod m
 	\end{align*}
 	and 
 	\begin{align*}
 	p_{2at,at}(an+b)\equiv 0 \pmod m
 	\end{align*}
 	for all non-negative integer $n$.
 \end{theorem}
 \begin{proof}
 Let $n\geq 0$. From \eqref{thm1.1}, we obtain 
 	\begin{align}\label{eqn-NEW-2}
 	&p_{at,at}(an+b)\notag\\
 	&=p(an+b)+\sum_{r=1}^{\infty}p(a(n-tr(2r+1))+b)-\sum_{s=1}^{\infty}p(a(n-ts(2s-1))+b).
 	\end{align}
 	We note that the terms remaining in the sums in \eqref{eqn-NEW-1} satisfy that $n-tr(2r+1)$ and $n-ts(2s-1)$ are non-negative. Hence, the same is true in \eqref{eqn-NEW-2}. Now, if $p(\ell a+b)\equiv 0\pmod{m}$ for every non-negative integer $\ell$, then \eqref{eqn-NEW-2} yields that $p_{at, at}(an+b)\equiv 0\pmod{m}$. This completes the proof of the first congruence of the theorem. 
 	\par Using \eqref{thm1.2} and proceeding along similar lines, we prove the second congruence. This completes the proof of the theorem.	
 \end{proof}
 As an application of Theorem \ref{thm2}, we find that $p_{t,t}(n)$ and $p_{2t, t}(n)$ satisfy the Ramanujan's famous congruences for certain infinite families of $t$. Much to Ramanujan's credit, the ``Ramanujan congruences'' for $p(n)$ are given below. If $k\geq 1$, then for every non-negative integer $n$, we have 
 \begin{align*}
 p\left(5^{k} n+\delta_{5, k}\right) & \equiv 0 \pmod {5^{k}}; \\
 p\left(7^{k} n+\delta_{7, k}\right) & \equiv 0 \pmod {7^{[k / 2]+1}}; \\
 p\left(11^{k} n+\delta_{11, k}\right) & \equiv 0 \pmod{ 11^{k}},
 \end{align*}
 where $\delta_{p, k}:= 1/24\pmod {p^k}$ for $p= 5, 7,11$.  In the following, we prove that $p_{at,at}(n)$ and $p_{2at,at}(n)$ satisfy the Ramanujan congruences when $a=5^k, 7^k, 11^k$.
 \begin{corollary}\label{ramanujan-cong}
 For all $k, t\geq 1$ and for every non-negative integer $n$, we have  
 \begin{align*}
 p_{5^{k}t,5^{k}t}\left(5^{k} n+\delta_{5, k}\right) & \equiv 0 \pmod {5^{k}}; \\
 p_{7^{k}t,7^{k}t}\left(7^{k} n+\delta_{7, k}\right) & \equiv 0 \pmod {7^{[k / 2]+1}}; \\
 p_{11^{k}t,11^{k}t}\left(11^{k} n+\delta_{11, k}\right) & \equiv 0 \pmod{ 11^{k}};\\
 p_{2\cdot5^{k}t,5^{k}t}\left(5^{k} n+\delta_{5, k}\right) & \equiv 0 \pmod {5^{k}}; \\
 p_{2\cdot7^{k}t,7^{k}t}\left(7^{k} n+\delta_{7, k}\right) & \equiv 0 \pmod {7^{[k / 2]+1}}; \\
 p_{2\cdot11^{k}t,11^{k}t}\left(11^{k} n+\delta_{11, k}\right) & \equiv 0 \pmod{ 11^{k}}.
 \end{align*}
 \end{corollary}
\begin{proof}
Combining Ramanujan congruences for $p(n)$ and Theorem \ref{thm2} we readily obtain that $p_{at,at}(n)$ and $p_{2at,at}(n)$ satisfy the Ramanujan congruences when $a=5^k, 7^k, 11^k$. This completes the proof.
\end{proof}
 \begin{remark}
 In \cite{sellers}, da Silva and Sellers proposed to undertake a more extensive investigation of the properties of $p_{2t,t}(n)$. Based on extensive computations, they remarked that $p_{2t,t}(n)$ need not satisfy Ramanujan-like congruences. But Corollary \ref{ramanujan-cong} disproves their observation. We have also verified Corollary \ref{ramanujan-cong} numerically for large values of the parameters.
 \end{remark}
 %%%%%%%%%%%%%%%%%%%%%%%%%%%%%%%%%%%%%%%
 \section{Mex-related partitions and relations to singular overpartitions}
 In this section we relate the mex-related partition functions to the Andrews' singular overpartion functions. This helps us to find new congruences satisfied by the mex-related partition functions. In the following theorem, we prove that both $p_{t,t}(n)$ and $\overline{C}_{4t,t}(n)$ have the same parity. 
 \begin{theorem}\label{thm3}
 	Let $t$ be a positive integer. Then, for all $n\geq 0$, we have 
 	\begin{align*}
 	p_{t,t}(n)\equiv \overline{C}_{4t,t}(n) \pmod 2.
 	\end{align*}
 \end{theorem}
 \begin{proof}
 	From \eqref{gen-fun}, we have
 	\begin{align}\label{thm3-new1}
 	\sum_{n=0}^{\infty} p_{t,t}(n)q^{n}=\frac{1}{(q; q)_{\infty}} \sum_{n=0}^{\infty}(-1)^{n}q^{t n(n+1)/2}.
 	\end{align}
 	Employing the Ramanujan's theta function
 	\begin{align*}
 	\psi(q):=\sum_{n=0}^{\infty} q^{n(n+1) / 2}=\frac{\left(q^{2} ; q^{2}\right)_{\infty}^{2}}{(q ; q)_{\infty}}
 	\end{align*}
 	into \eqref{thm3-new1}, we find that
 	\begin{align}\label{thm3.2}
 	\sum_{n=0}^{\infty} p_{t,t}(n) q^{n} &\equiv \frac{1}{(q; q)_{\infty}} \sum_{n=0}^{\infty} q^{t n(n+1)/2}\pmod{2}\nonumber\\
 	&=\frac{1}{(q ; q)_{\infty}} \frac{\left(q^{2t} ; q^{2t}\right)_{\infty}^{2}}{\left(q^t; q^t\right)_{\infty}}\nonumber \\
 	&\equiv \frac{1}{(q ; q)_{\infty}} \frac{\left(q^{t} ; q^{t}\right)_{\infty}^{4}}{\left(q^{t} ; q^{t}\right)_{\infty}}\pmod{2}\nonumber\\
 	&\equiv\frac{\left(q^{t} ; q^{t}\right)_{\infty}^{3}}{(q ; q)_{\infty}}\pmod{2}. 
 	\end{align}
 	From \eqref{gen-fun-sop}, we find that 
 	\begin{align}\label{thm3.3}
 	\sum_{n=0}^{\infty} \overline{C}_{4t, t}(n) q^{n} &=\frac{\left(q^{4t} ; q^{4t}\right)_{\infty}\left(-q^{t} ; q^{4 t}\right)_{\infty}\left(-q^{3t} ; q^{4t}\right)_{\infty}}{(q ; q)_{\infty}}\nonumber \\
 	&=\frac{\left(q^{4t} ; q^{4t}\right)_{\infty}\left(-q^{t} ; q^{4t}\right)_{\infty}\left(-q^{3t} ; q^{4 t}\right)_{\infty}\left(-q^{2t} ; q^{4t}\right)_{\infty}\left(-q^{4t} ; q^{4t}\right)_{\infty}}{(q ; q)_{\infty}\left(-q^{2t} ; q^{4t}\right)_{\infty}\left(-q^{4t} ; q^{4t}\right)_{\infty}} \nonumber \\
 	&=\frac{\left(q^{4t} ; q^{4t}\right)_{\infty}\left(-q^{t} ; q^{t}\right)_{\infty}}{(q ; q)_{\infty}\left(-q^{2t} ; q^{2t}\right)_{\infty}} \nonumber\\
 	&=\frac{\left(q^{4t} ; q^{4t}\right)_{\infty}\left(q^{2t} ; q^{2t}\right)_{\infty}^{2}}{(q ; q)_{\infty}\left(q^{t} ; q^{t}\right)_{\infty}\left(q^{4t} ; q^{4t}\right)_{\infty}} \nonumber\\
 	& \equiv \frac{\left(q^{t} ; q^{t}\right)_{\infty}^{3}}{(q ; q)_{\infty}} \pmod 2.
 	\end{align}
 	Thus, combining \eqref{thm3.2} and \eqref{thm3.3}, we have
 	\begin{align*}
 	\sum_{n=0}^{\infty}p_{t,t}(n)\equiv \sum_{n=0}^{\infty}\overline{C}_{4t,t}(n)\pmod 2.
 	\end{align*}	
 	This completes the proof of the theorem.
 \end{proof}
In \cite{chen2015}, Chen, Hirschhorn and Sellers proved that, for all $n\geq 1$,
\begin{align*}
\overline{C}_{4, 1}(n)=\begin{cases}
1 \pmod{2}, & \mbox{if $n=k(3k- 1)$ for some $k$};\\
0\pmod{2}, & \mbox{otherwise}.\nonumber
\end{cases}
\end{align*} 
 Due to Theorem \ref{thm3}, we have the same parity characterization for $p_{1,1}(n)$. Recently, da Silva and Sellers also found the same parity characterization for $p_{1,1}(n)$ (see for example \cite[Theorem 4]{sellers}). To the best of our knowledge, the parity characterization for $\overline{C}_{12, 3}(n)$ is not known till date. In \cite{sellers}, da Silva and Sellers found the parity characterization for $p_{3,3}(n)$ (see for example \cite[Theorem 7]{sellers}). Combining \cite[Theorem 7]{sellers} and Theorem \ref{thm3}, we have the following parity characterization for $\overline{C}_{12, 3}(n)$:
 \begin{corollary}
 	For all $n \geq 1,$ we have
 	\begin{align*}
 	\overline{C}_{12, 3}(n)=\begin{cases}
 	1 \pmod{2}, & \mbox{if $3n+1$ is a square};\\
 	0\pmod{2}, & \mbox{otherwise}.\nonumber
 	\end{cases}
 	\end{align*}
 \end{corollary}
 %%%%%%%%%%%%%%%%%%%%%%%%%%%%%%%%%%%%%%%%%
 We next use knwon congruences for Andrews' singular overpartition functions $\overline{C}_{4t,t}(n)$ and combine them with Theorem \ref{thm3} to deduce new congruences for $p_{t,t}(n)$ for different values of $t$. Our list of congruences obtained this way need not be exhaustive.
 \begin{theorem} \label{thm5}
 	Let $p \geq 5$ be a prime and $p \not\equiv 1 \pmod {12}$. Then, for all $k, n \geq 0$ with $p \nmid n$, we have 
 	\begin{align*}
 	p_{1,1}\left(p^{2 k+1} n+\frac{p^{2 k+2}-1}{12}\right) \equiv 0 \pmod 2.
 	\end{align*}
 \end{theorem}
 \begin{proof}
 	Taking $t=1$ in Theorem \ref{thm3}, we have
 	\begin{align}\label{thm5.1}
 	p_{1,1}(n)\equiv \overline{C}_{4,1}(n)\pmod 2.
 	\end{align}
 	Thanks to Chen, Hirschhorn and Sellers (\cite[Corollary 3.6.]{chen2015}) we know that, for all $n\geq 0$, 
 	\begin{align}\label{thm5.2}
 	\overline{C}_{4,1}\left(p^{2 k+1} n+\frac{p^{2 k+2}-1}{12}\right) \equiv 0 \pmod 4.
 	\end{align}
 Combining \eqref{thm5.1} and \eqref{thm5.2}, we deduce the required congruence.
 \end{proof}
 %%%%%%%%%%%%%%%%%%%%%%%%%%%%%%%%%%%%%%%%%
 \begin{theorem}\label{thm11}
 	If $p$ is a prime such that $p \equiv 3\pmod 4$ and $1 \leq j \leq p-1,$ then for all non-negative integers $\alpha$ and $n,$ we have
 	\begin{align*}
 	p_{2,2}\left(p^{2 \alpha+1}(p n+j)+\frac{5\left(p^{2(\alpha+1)}-1\right)}{24}\right) \equiv 0\pmod 2.
 	\end{align*}	
 \end{theorem}
 \begin{proof}
 	In \cite[Theorem 1.7]{ahmed2015}, Ahmed and Baruah proved that, 
 	if $p$ is a prime such that $p \equiv 3\pmod 4$ and $1 \leq j \leq p-1$, then for all non-negative integers $\alpha$ and $n$, 
 	\begin{align*}
 	\overline{C}_{8,2}\left(p^{2 \alpha+1}(p n+j)+\frac{5\left(p^{2(\alpha+1)}-1\right)}{24}\right) \equiv 0\pmod 2.
 	\end{align*}	
 Now, Theorem \ref{thm3} yields that the same congruence is also satisfied by $p_{2,2}(n)$. 
 \end{proof}
 %%%%%%%%%%%%%%%%%%%%%%%%%%%%%
 In the following theorem, we find congruences satisfied by $p_{3,3}(n)$.
 \begin{theorem}\label{thm6-NEW}
 	We have:
 	\begin{enumerate} 
 	\item For all $n \geq 0$,
 	\begin{align} \label{thm6}
 	p_{3,3}(16 n+11) \equiv p_{3,3}(16 n+15) \equiv 0\pmod 2.
 	\end{align}
 \item If $n$ can not be represented as the sum of a pentagonal number and four times a pentagonal number, then
 	\begin{align}\label{thm7}
 	p_{3,3}(16n+3) \equiv 0\pmod 2.
 	\end{align}
 \item  If $n$ can not be represented as the sum of two times a pentagonal number and three times a triangular number, then
 	\begin{align}\label{thm8}
 	p_{3,3}(16 n+7) \equiv 0\pmod 2.
 	\end{align}
 \end{enumerate}
\end{theorem}
 %%%%%%%%%%%%%%%%%%%%%%%%%%%
 \begin{proof}
 	Taking $t=3$ in Theorem \ref{thm3}, for all $n\geq 0$, we have 	
 	\begin{align}\label{thm.1}
 	p_{3,3}(n)\equiv \overline{C}_{12,3}(n) \pmod 2.
 	\end{align}
 	By Theorems 4.1, 4.2 and 4.3 of Li and Yao \cite{Li2018} we know that, for all $n\geq 0$, 
 	\begin{align}\label{thm.2}
 	\overline{C}_{12,3}(16 n+11) \equiv \overline{C}_{12,3}(16 n+15) \equiv 0 \pmod 8;
 	\end{align}
 	\begin{align}\label{thm.3}
 	\overline{C}_{12,3}(16 n+3) \equiv 0 \pmod 8;
 	\end{align}
 	and
 	\begin{align}\label{thm.4}
 	\overline{C}_{12,3}(16 n+7) \equiv 0 \pmod 8.
 	\end{align}
 Combining \eqref{thm.1}, \eqref{thm.2}, \eqref{thm.3} and \eqref{thm.4} we complete the proof of the theorem.
 \end{proof}
 \begin{corollary}\label{cor1}
 	We have:
 	\begin{enumerate}
 	\item Let $p \geq 5$ be a prime with $p \equiv 3\pmod 4$. For $\alpha, n \geq 0,$ if $p \nmid n$ then
 	\begin{align*}
 	p_{3,3}\left(16 p^{2 \alpha+1} n+\frac{10 p^{2 \alpha+2}-1}{3}\right) \equiv 0\pmod 2.
 	\end{align*}
 	\item Let $p \geq 5$ be a prime with $\left(\frac{-2}{p}\right)=-1$. For $\alpha, n \geq 0$, if $p \nmid n$ then
 	\begin{align*}
 	p_{3,3}\left(16 p^{2 \alpha+1} n+\frac{22 p^{2 \alpha+2}-1}{3}\right) \equiv 0\pmod 2.
 	\end{align*}
 	\end{enumerate}
 \end{corollary}
 \begin{proof}
 	By Corollary 4.2 and Corollary 4.3 of Li and Yao \cite{Li2018}, for all $n$,  we have 
 	\begin{align}\label{cor.1}
 	\overline{C}_{12,3}\left(16 p^{2 \alpha+1} n+\frac{10 p^{2 \alpha+2}-1}{3}\right) \equiv 0\pmod 8
 	\end{align}
 	and 
 	\begin{align}\label{cor.2}
 	\overline{C}_{12,3}\left(16 p^{2 \alpha+1} n+\frac{22 p^{2 \alpha+2}-1}{3}\right) \equiv 0\pmod 8.
 	\end{align}
 Now, combining \eqref{thm.1}, \eqref{cor.1} and \eqref{cor.2} we complete the proof.
 \end{proof}
 %%%%%%%%%%%%%%%%%%%%%%%%%%%%%%%%%%%%%%%%%
 In \cite{utpal2018}, Pore and Fathima found congruences for $\overline{C}_{20, 5}(n)$. Combining their results and Theorem \ref{thm3} for $t=5$, we obtain the following two theorems.
 \begin{theorem}\label{thm12}
 	For all $\alpha, n \geq 0$, we have 
 	\begin{align}
 	\label{thm12.1} p_{5,5}\left(2 \cdot 5^{2 \alpha+1} n+\frac{31 \cdot 5^{2 \alpha}-7}{12}\right) \equiv 0 \pmod 2; \\
 \label{thm12.2}	 p_{5,5}\left(2 \cdot 5^{2 \alpha+1} n+\frac{79 \cdot 5^{2 \alpha}-7}{12}\right) \equiv 0 \pmod 2; \\
 	p_{5,5}\left(2 \cdot 5^{2 \alpha+2} n+\frac{83 \cdot 5^{2 \alpha+1}-7}{12}\right) \equiv 0 \pmod 2; \notag \\
 	p_{5,5}\left(2 \cdot 5^{2 \alpha+2} n+\frac{107 \cdot 5^{2 \alpha+1}-7}{12}\right) \equiv 0 \pmod 2.\notag 
 	\end{align}
 \end{theorem}
\begin{proof}
Combining \cite[Theorem 1.6]{utpal2018} and Theorem \ref{thm3} for $t=5$, we obtain the desired congruences satisfied by $p_{5,5}(n)$.
\end{proof}
 \begin{theorem}\label{thm13}
 	Let $p \geq 5$ be a prime such that $\left(\frac{-10}{p}\right)=-1$ and $j=1, 2, \ldots, p-1$. Then, for all $\alpha, n \geq 0$,
 	\begin{align*}
 	p_{5,5}\left(2 p^{2 \alpha+1}(p n+j)+7 \times \frac{p^{2 \alpha+2}-1}{12}\right) \equiv 0 \pmod 2.
 	\end{align*}
 \end{theorem}
 \begin{proof}
 Combining \cite[Theorem 1.7]{utpal2018} and Theorem \ref{thm3} for $t=5$, we obtain the desired congruences satisfied by $p_{5,5}(n)$.
 \end{proof}
 \begin{theorem}\label{thm14}
 	For all $\alpha\geq 0,$ we have
 	\begin{align}\label{thm14.1}
 	p_{7,7}\left(2 \cdot 7^{2 \alpha+1} n+\frac{(11+12 r) \cdot 49^{\alpha}-5}{6}\right) \equiv 0\pmod 2,~~r\in\{3,4,6\}
 	\end{align}
 	and
 	\begin{align}\label{thm14.2}
 	p_{7,7}\left(2 \cdot 49^{\alpha+1} n+\frac{(12 s+5) \cdot 7^{2 \alpha+1}-5}{6}\right) \equiv 0\pmod 2,~~s\in\{2,4,5\}
 	\end{align}
 \end{theorem}
 \begin{proof}
 In the proof of Theorem 5.1, Li and Yao \cite{Li2018} found two congruences satisfied by $\overline{C}_{28,7}(n)$, for example see  \cite[(5.22) \& (5.23)]{Li2018}. Combining these two congruences and Theorem \ref{thm3} for $t=7$, we complete the proof of the theorem.	
 \end{proof}
 \begin{remark}
 	In \cite[Theorem 11]{sellers}, da Silva and Sellers found relations between $p_{t,t}(n)$ and the $t$-core partition functions. They used certain congruences for $t$-core partition functions (obtained by Radu and Sellers \cite{Radu-sellers} using modular forms) to find several congruences satisfied by $p_{t,t}(n)$ when $t=5, 7, 11, 13, 17, 19, 23$. They also proposed to find a fully elementary proof of their congruences listed in Theorem 11.  Putting $\alpha=0$ in \eqref{thm12.1} and \eqref{thm12.2}, we obtain the congruences for $p_{5,5}(n)$ listed in \cite[Theorem 11]{sellers}. Again, putting $\alpha=0$ in \eqref{thm14.1}, we obtain the congruences for $p_{7,7}(n)$ listed in \cite[Theorem 11]{sellers}.
 \end{remark}
 %%%%%%%%%%%%%%%%%%%%%%%%%%%%%%%%%%%%%%%%%
 \begin{theorem}
 	Let $p \geq 5$ be a prime with $\left(\frac{-21}{p}\right)=-1$. We have:
 	\begin{enumerate}
 	\item For $n, \alpha, \beta \geq 0$,
 	\begin{align*}
 	p_{7,7}\left(2 \cdot 7^{2 \alpha+1} p^{2 \beta} n+\frac{(11+12 r) \cdot 49^{\alpha} p^{2 \beta}-5}{6}\right) \equiv 0\pmod 2
 	\end{align*}
 	and
 	\begin{align*}
 	p_{7,7}\left(2 \cdot 49^{\alpha+1} p^{2 \beta} n+\frac{(5+12 s) \cdot 7^{2 \alpha+1} p^{2 \beta}-5}{6}\right) \equiv 0\pmod 2,
 	\end{align*}
 	where $r \in\{3,4,6\}$ and $s \in\{2,4,5\}$.
 \item For $n, \alpha, \beta \geq 0$, if $p \nmid n$ then
 	\begin{align*}
 	p_{7,7}\left(2 \cdot 49^{\alpha} p^{2 \beta+1} n+\frac{11 \cdot 49^{\alpha} p^{2 \beta+2}-5}{6}\right) \equiv 0\pmod 2.
 	\end{align*}
 	\end{enumerate}
 \end{theorem}
 %%%%%%%%%%%%%%%%%%%%%%%%%%%%%
 \begin{proof}
 	Li and Yao (\cite[Theorem 5.1 and 5.2 ]{Li2018}) proved that, for all $n\geq 0$, 
 	\begin{align}\label{thm9.2}
 	\overline{C}_{28,7}\left(2 \cdot 7^{2 \alpha+1} p^{2 \beta} n+\frac{(11+12 r) \cdot 49^{\alpha} p^{2 \beta}-5}{6}\right) \equiv 0\pmod 4;
 	\end{align}
 	\begin{align}\label{thm9.3}
 	\overline{C}_{28,7}\left(2 \cdot 49^{\alpha+1} p^{2 \beta} n+\frac{(5+12 s) \cdot 7^{2 \alpha+1} p^{2 \beta}-5}{6}\right) \equiv 0\pmod 4
 	\end{align}
 	and
 	\begin{align}\label{thm10.1}
 	\overline{C}_{28,7}\left(2 \cdot 49^{\alpha} p^{2 \beta+1} n+\frac{11 \cdot 49^{\alpha} p^{2 \beta+2}-5}{6}\right) \equiv 0\pmod 4.
 	\end{align}
 Combining \eqref{thm9.2}, \eqref{thm9.3}, \eqref{thm10.1}, and Theorem \ref{thm3} with $t=7$, we find the desired congruences satisfied by $p_{7,7}(n)$.
 \end{proof}
 %%%%%%%%%%%%%%%%%%%%%%%%%%%%%%%

%\bibliographystyle{acm}
%\bibliographystyle{sej}
%\bibliographystyle{plain}
%\bibliography{ref.cray.bib}
\end{document}